\documentclass[11pt,amsfonts]{article}
\usepackage{amsmath,amsfonts,amsthm,amssymb,color}
\usepackage[T1]{fontenc}
\usepackage{enumerate}
\usepackage{graphicx}
\usepackage{float}
\usepackage[ruled,vlined]{algorithm2e}
\usepackage{amssymb}
\usepackage{amsmath}
\usepackage{color}
\usepackage{layout}
\usepackage{amsfonts}
\usepackage{dsfont}
\usepackage{cite}
\usepackage{color}
\usepackage{listings}
\usepackage{anysize}
\usepackage[noend]{algpseudocode}
\usepackage{ulem}
\usepackage{latexsym}
\usepackage{ragged2e}
\usepackage{pdfsync}

\newtheorem{prop}{Proposition}
\newtheorem{lemma}{Lemma}

\newtheorem{corollary}{Corollary}
\newtheorem{theorem}{Theorem}
\newtheorem{remark}{Remark}

\def\real{{\mathord{{\rm I\kern-2.8pt R}}}}        
\def\inte{{\mathord{{\rm I\kern-2.8pt N}}}}

\def\sZZ{{\rm Z\kern-2.8ptem{}Z}}

\def\z{{\mathchoice
  {\sZZ}
  {\sZZ}
  {\rm Z\kern-0.30em{}Z}
  {\rm Z\kern-0.25em{}Z} }}
\def\sQQ{{\kern 0.27em \vrule height1.45ex width0.03em depth0em
          \kern-0.30em \rm Q}}
\def\qu{{\mathchoice
    {\sQQ}
    {\sQQ}
  {\kern 0.225em \vrule height1.05ex width0.025em depth0em \kern-0.25em \rm Q}
  {\kern 0.180em \vrule height0.78ex width0.020em depth0em \kern-0.20em \rm Q}
        }}
\def\sCC{{\kern 0.27em \vrule height1.45ex width0.03em depth0em
          \kern-0.30em \rm C}}
\def\complex{{\mathchoice
    {\sCC}
    {\sCC}
  {\kern 0.225em \vrule height1.05ex width0.025em depth0em \kern-0.25em \rm C}
  {\kern 0.180em \vrule height0.78ex width0.020em depth0em \kern-0.20em \rm C}
        }}

\font\tenmath=msbm10 \font\sevenmath=msbm7 \font\fivemath=msbm5
\newfam\mathfam \textfont\mathfam=\tenmath
\scriptfont\mathfam=\sevenmath \scriptscriptfont\mathfam=\fivemath

\newcommand{\ignore}[1]{}
\begin{document}

\renewcommand{\thefootnote}{\fnsymbol{footnote}}

\title{Gamma mixed fractional Lévy Ornstein - Uhlenbeck process}

\author{ H\'ector Araya$^{1}$  \ \ \ \ Johanna Garz\'on$^{2}$  \ \ \ \  Rolando Rubilar$^{3}$  \\
\small $^{1}$ Instituto de Estad\'istica, Facultad de Ciencias, Universidad de Valpara\'iso,\\
  hector.arayaca@uv.cl\\ 
\small $^{2}$ Departamento de Matem\'atica, Facultad de Ciencias, Universidad Nacional de Colombia,\\
mjgarzonm@unal.edu.co\\ 
\small $^{3}$ Instituto de Estad\'istica, Facultad de Ciencias, Universidad de Valpara\'iso,\\
  rolando.rubilar@uv.cl\\ 
}

\maketitle
\begin{abstract}
In this article, we introduce  a non Gaussian long memory process constructed by the aggregation of independent copies of a fractional Lévy Ornstein-Uhlenbeck process with random coefficients. Several properties and a limit theorem are studied for this new process. Finally, some simulations of the limit process are shown.    
\end{abstract}
\vskip0.3cm

{\bf 2010 AMS Classification Numbers:} 60G10,  60G17, 60H05, 60H30.

\section{Introduction}

An Ornstein-Uhlenbeck (OU) process is a diffusion process introduced by the physicists Leonard Salomon Ornstein and George Eugene Uhlenbeck \cite{OU} to describe the stochastic behavior of the velocity of a particle undergoing Brownian motion. The OU process $X=\{X(t), \  {t\geq 0}\}$ is  the solution of the Langevin equation 
\begin{equation}
\label{OU}
    dX(t)= \alpha X(t) dt + \sigma dB(t), \ \ t \geq 0,
\end{equation}
where $X(0)=x\in \mathbb{R}$, $B=\{B(t), t\geq 0\}$ is a Brownian motion and $\alpha, \sigma$ are constants. This process is  stationary, Gaussian and Markovian; in fact, it is the only stochastic process which has all these three properties. This process is used for modelling in many different fields such as physics, biology and finance among others (see \cite{Aalen, butler, kahl, OU, ricciardi, Vasicek} and references therein) and it has been widely generalized. 

Different extensions of the Ornstein-Uhlenbeck processes have been obtained to replace the Brownian motion in \eqref{OU} by more general noise processes; for example, Lévy OU \cite{Barndorff}, fractional OU \cite{cheridito}, sub-fractional OU  \cite{mishura} or Hermite OU processes \cite{maejima}. These  are introduced by solution of the Langevin equation with driving noise given by a Lévy process, fractional Brownian motion, sub-fractional Brownian motion, Hermite process respectively. 

In the study of long-range dependence, Igoli and Terdik \cite{igloi} defined a generalization of OU  process with this property. This process is called Gamma-mixed Ornstein-Uhlenbeck processes and it is  built via aggregation of a sequence of random coefficient independent Ornstein-Uhlenbeck processes. Let us be more precise. Given a sequence $(X_k)_{k\in \mathbb{N}}$ of stochastic processes such that for each $k\geq 1$, the process $X_k$ is the solution of the Langevin equation 

\begin{equation}
\label{OU1}
    dX_k(t)= \alpha_k X(t) dt + dB(t), 
\end{equation}
where $B$ is a Brownian motion with time parameter $t \in \mathbb{R}$ and $(-\alpha_k)_{k \in \mathbb{N}}$ are independent random variables (also independent of $B$) with Gamma distribution $\Gamma(1-h, \lambda)$ with $h\in (0,1)$ and $\lambda >0$. The aggregated process is given by
$$Y_n(t) = \frac{1}{n}\sum_{k=1}^n X_k(t),$$
and it converges, as $n \to \infty$, to a stochastic process $Y$ which is a stationary Gaussian process, semimartingale, asymptotically self-similar and it has long-range dependence. This limit process is the so-called Gamma-mixed Ornstein-Uhlenbeck. In a similar way, in \cite{dou} and \cite{es1} the authors  studied the case where $B$ in \eqref{OU1} is generalized to be a fractional Brownian motion and Hermite process respectively; they define the fractional Ornstein-Uhlenbeck process mixed with a Gamma distribution and the Hermite Ornstein-Uhlenbeck process mixed with a Gamma distribution, both processes exhibit long range dependence, the first one is a Gaussian process but the second one is not Gaussian.

The aim of this paper is to define and study some properties of the Gamma mixed fractional Lévy Ornstein-Uhlenbeck process obtained as the limit of the aggregated OU process driven by fractional Lévy process and random coefficient with Gamma distribution. 

The fractional Lévy process (fLp) was  defined by \cite{mar} as a  generalization of the moving average representation of fractional Brownian motion given by Mandelbrot and Van Ness \cite{mandelbrot} when replacing the Brownian motion in this integral representation by a  Lévy process with zero mean, finite variance and without Gaussian part. FLp is almost surely H\"olderian, has stationary increments and long range dependence,  but unlike fractional Brownian motion, this process is neither Gaussian nor self-similar process.  In \cite{Fink}, authors introduced the fractional Lévy Ornstein‐Uhlenbeck process  (fLOUp) as the unique stationary pathwise solution of the Langevin equation driven by a fLp and prove that its increments exhibit long range dependence. Recently, many authors have studied fLp and the fractional Lévy Ornstein‐Uhlenbeck process on theoretical and applicable levels, see for example \cite{bender, bender2, glaser, klau, Shen, tika, wang} and the references therein.

This paper is organized as follows.  In Section 2 we give a brief introduction to the fLp and the stochastic calculus related to this. Fractional Lévy Ornstein‐Uhlenbeck  process with random coefficient is introduced in Section 3.  In Section 4,  we define the aggregated processwhich of fLOUp with random coefficient  and study  its limit process, which we will call  Gamma mixed fractional Lévy Ornstein-Uhlenbeck process. Finally, in Section 5 we present some simulations of the paths of the Gamma mixed fractional Lévy Ornstein-Uhlenbeck process.

\section{Preliminaries}\label{Pre}
In this section, we briefly recall some relevant aspects of the fractional Lévy process (fLp), its main properties and stochastic integrals with respect to this  fLp. This process will be used in the remainder of the paper. We work on a complete probability space  $(\Omega_{L}, {\cal F}_{L}, \mathbb{P}_{L})$.
\subsection{Fractional Lévy process}
The fractional Lévy process $L^{d}=\{L^{d}_{t}, \ t \in \mathbb{R}\}$, with $d  \in ( 0, 1/2 ) $, is a non-Gaussian process defined as follows (see \cite{mar}): 

\begin{equation}\label{poisson}
L_t^d=\int_{\mathbb{R}} f_t^{(d)}(s) dL(s),  \quad t \in \mathbb{R},
\end{equation}
where the kernel function $f_t^{(d)}$ is given by
\begin{equation}\label{Kernel}
f_t^{(d)}(s) = {1 \over \Gamma(d +1) } [(t-s)_{+}^{d} - (-s)_{+}^{d}],  \ \ \ s\in\mathbb{R}
\end{equation}
and $L=\{L(t), \ t \in \mathbb{R}\}$ is a zero-mean two-sided L\'evy process with $\mathbb{E}(L_{1}^{2}) < \infty $ and without  Brownian component, i.e. 
\begin{equation*}
L_{t} = L^{(1)}_{t}1_{\lbrace t \geq 0 \rbrace} -L^{(2)}_{-t}1_{\lbrace t \leq 0 \rbrace},
\end{equation*}
where $L^{(1)}$ and $L^{(2)}$ are two independent copies of the same one-sided Lévy process. 

The following Lemma (see \cite{mar, klau}) establishes that the fLp  is well defined in the $L^2(\Omega)$-sense and gives its characteristic function.
\begin{lemma}\label{lemma3}
Let $L=\{L_t\}_{t\in \mathbb{R}}$ be a two-sided Lévy process without Brownian component such that $E[L(1)]=0$ and $E[L(1)^{2}]<\infty$. For $t \in \mathbb{R}$, let $f_{t}\in L^{2} (\Omega)$. Then the integral $S(t):=\int_{\mathbb{R}} f_{t}(u) dL(u)$ exists in the $L^{2} (\Omega)$ sense and $E[S(t)] = 0$.  Furthermore, $S(t)$ satisfies the isometry  
\begin{equation*}
E[(S(t))^{2}] =E[(L(1))^{2}] \Vert  f_{t}(\cdot )\Vert_{L^{2}(\mathbb{R})}, \ \ \ t \in \mathbb{R}, 
\end{equation*}
the covariance function of process $S$ is given by
\begin{equation*}
\tilde{\Gamma}(s,t) = cov(S(s),S(t)) =E[(L(1))^{2}]  \int_{\mathbb{R}} f_{t}(u ) f_{s}(u) du, \ \ s,t \in \mathbb{R}
\end{equation*}
and  the characteristic function of  $S(t_{1}), \ldots , S(t_{m})$ for $t_{1} < \ldots < t_{m} $ and $m \in \mathbb{N}$ is given by
\begin{equation*}
E_{L}\left[ \exp\left(  i \sum_{j=1}^{m} \theta_{j} S(t_{j}) \right) \right] =  \exp\left(  \int_{\mathbb{R}} 
\psi\left(      \sum_{j=1}^{m} \theta_{j} f_{t_{j}}(s)   \right) ds \right), \ \ \ \text{for} \ \theta_{j} \in \mathbb{R}, \ j=1, \ldots , m,
\end{equation*}
where 
$$\psi(u) = \int_{\mathbb{R}} (e^{iux}-1-iux) \nu(dx), \quad u \in \mathbb{R}$$ 
and $\nu$ is the Lévy measure of $L$.
\end{lemma}

From (\ref{poisson}) we can see that the covariance function of $L^d$ is given by
\begin{equation}
\mathbb{E}(L^{d}_{t} L^{d}_{s})= \frac{1}{2}V_{d}^{2}\left( \vert t \vert^{2d+1} +   \vert s \vert^{2d+1} - \vert t-s \vert^{2d+1} \right), \quad t,s \in \mathbb{R},
\end{equation}
where $V_{d}^{2}= \frac{\mathbb{E}(L_{1}^{2})}{2 \Gamma (2d+2) \sin ( \pi (d + 1/2) )}$. Up to a scaling constant, this is the same covariance as fractional Brownian motion.\\

The fractional Lévy process $L^d$ defined by (\ref{poisson})  has the following properties (see \cite{mar}  for their proofs):  
\begin{itemize}
\item For any $\beta \in (0, d)$,  the sample paths of $L^d$ are a.s $\beta$- H\"older continuous.
\item $L^{d}$ is a process with stationary increments and symmetric, i.e. $ \lbrace L^{d}_{-t} \rbrace_{t \in \mathbb{R}}  \stackrel{(d)}{=} \lbrace- L^{d}_{t} \rbrace_{t \in \mathbb{R}}$. 
\item $L^{d}$ cannot be self-similar. However,  $L^{d}$ is asymptotically self-similar with parameter $0<d<0.5$, i.e 
\begin{equation*}
\lim\limits_{c \rightarrow \infty} \left\lbrace \frac{L^{d}_{ct}}{c^{d}} \right\rbrace_{t \in \mathbb{R}}  \stackrel{(d)}{=}  \left\lbrace B^{d}_{t} \right\rbrace_{t \in \mathbb{R}},
\end{equation*}    
where the equality is in the sense of finite dimensional distributions.
\item For $h>0$, the covariance between two increments $L^{d}_{t+h} - L^{d}_{t}$ and  $L^{d}_{s+h} - L^{d}_{s}$, where $s+h \leq t$ and $t-s=nh$ is 
\begin{eqnarray}
\delta_{d}(n) &=& V_{d}^{2} h^{2d+1} \left[ (n+1)^{2d+1} + (n-1)^{2d+1} - 2n^{2d+1}\right] \nonumber \\
&=& V_{d}^{2} d(2d+1) h^{2d+1}n^{2d-1} + O(n^{2d-2}), \quad n \rightarrow \infty. 
\end{eqnarray} 
\item The increments of $L^d$ exhibit long memory, in the sense that for $d>0$, we have 
\begin{equation*}
\sum_{n=1}^{\infty} \delta_{d}(n) = \infty.
\end{equation*}

\end{itemize}
In the following, we recall two results from the reference \cite{mar} concerning  to stochastic integrals with respect to fractional Lévy process. 

Let $L^{1}(\mathbb{R})$ and $L^{2}(\mathbb{R})$ be the spaces of integrable and square integrable real functions, respectively and  $H$ the completion of $L^{1}(\mathbb{R}) \cap L^{2}(\mathbb{R})$  with respect to the norm $\Vert g \Vert^{2}_{H} = E[L_{1}^{2}] \int_{\mathbb{R}} (I^{d}_{-}g)^{2}(u) du$, where $(I^{d}_-{g})$ denotes the right-sided Riemann-Liouville fractional integral defined by 
$$(I^{d}_-{g})(x)= \dfrac{1}{\Gamma(d)} \int_{x}^{\infty} g(t)(t-x)^{d-1}dt.$$ 
\begin{lemma}\label{lemma1}
If $g \in H$, then 
$$\int_{\mathbb{R}}  g(s) dL^{d}_{s} = \int_{\mathbb{R}}  (I^{d}_-{g})(s)  dL_{s} ,$$
where the equality is in $L^{2}$ sense.
\end{lemma}
The next second-order property of the stochastic integral respect to fLp will be a key tool in this article.

\begin{lemma}\label{lemma2}
If $ \vert f \vert ,  \vert g \vert \in H$, then 
$$E \left[ \int_{\mathbb{R}}  f(s) dL^{d}_{s} \int_{\mathbb{R}}  g(s) dL^{d}_{s}  \right] = C_d \int_{\mathbb{R}}  \int_{\mathbb{R}} f(t) g(s) \vert  t-s\vert^{2d-1} dsdt, $$
\end{lemma}
where
\begin{equation}\label{eqcd} 
C_d =  \dfrac{\Gamma(1-2d)E[L_{1}^{2}]}{\Gamma(d) \Gamma(1-d)}.
\end{equation}

The following result from reference \cite{Fink}, established the solution of the Langevin equation driven by a fractional Lévy process. 

\begin{theorem}
\label{teofLOU}
Let $L^d$ be an fLp, $d \in (0, 1/2)$ and $\lambda > 0$. Then the unique stationary pathwise solution of the Langevin equation 
\begin{equation*}
\label{OU2}
    dX^{d, \lambda}(t)= \lambda X^{d, \lambda}(t) dt + dL^d_t 
\end{equation*}
is given a.s. by 
\begin{equation}
\label{eqfLOUp}
    X^{d, \lambda}(t) = \int_{-\infty}^{t} e^{\lambda(t-u) }dL^{d}(u), \quad t \in \mathbb{R},
\end{equation}
this process is called fractional Lévy-Ornstein-Uhlenbeck processes (fLOUp).
\end{theorem}

\section{Ornstein-Uhlenbeck process with random coefficient}

In this section, we study the fractional Lévy-Ornstein-Uhlenbeck processes (fLOUp) with random coefficients. First, we establish the existence of the solution, and then some properties of the process are shown.

\

We consider the fractional Lévy Ornstein-Uhlenbeck process $V^{d} = (V^{d}(t) , t \in \mathbb{R})$  given as the solution of
\begin{equation}\label{sde}
dV^{d}(t) = \lambda V^{d}(t)dt + dL^{d}(t), \quad t \in \mathbb{R} ,  
\end{equation}
where $L^{d}$ is a fLp with $d \in (0,1/2)$ and defined on $(\Omega_{L}, {\cal F}_{L}, \mathbb{P}_{L})$; and the coefficient $\lambda$ is a random variable defined on the probability space $(\Omega_{\lambda}, {\cal F}_{\lambda}, \mathbb{P}_{\lambda})$ and independent by $L^{d}$. We assume that $-\lambda$ follows a Gamma distribution with parameters $1-h$ and $\alpha$, i.e. $-\lambda \sim \Gamma(1-h, \alpha)$ with $h \in (0,1)$ and $\alpha >0$. 

\

From Theorem \ref{teofLOU}, the SDE given by  (\ref{sde}) has the explicit solution 
\begin{equation}\label{v}
V^{d}(t) = \int_{-\infty}^{t} e^{\lambda (t-u) }dL^{d}(u), \quad t \in \mathbb{R}, 
 \end{equation}
 where the initial condition is given by 
 \begin{equation*}
V^{d}(0) = \int_{-\infty}^{0} e^{-\lambda u }dL^{d}(u).
 \end{equation*}
 Clearly, this process is stationary. Furthermore, by Lemma \ref{lemma1} we can get 
 \begin{equation*}
V^{d}(t) =  \dfrac{1}{\Gamma(d)} \int_{\mathbb{R}} \int_{-\infty}^{t} e^{\lambda (t-v) } (v-u)_{+}^{d-1} dv dL(u), \quad t \in \mathbb{R}.
 \end{equation*}
 
 \
 
We will prove that for every $\omega_{\lambda} \in \Omega_{\lambda}$, the process $V^d$  is well defined in $L^{2}(\Omega_{L})$. In fact,  by Lemma \ref{lemma2} we have that for $C_{d} = \dfrac{\Gamma(1-2d)}{\Gamma(d) \Gamma(1-d)} E[L_{1}^{2}]$
\begin{eqnarray}
\label{eqEV2}
E_{L}[(V^{d}(t))^{2}] &=& C_{d} \int_{-\infty}^{t} \int_{-\infty}^{t}  e^{\lambda (t-u) } e^{\lambda_{k}(t-v) } \vert u-v \vert^{2d-1} du dv \notag \\
&=& C_{d} \int_{0}^{\infty} \int_{0}^{\infty}  e^{\lambda (u+v) }  \vert u-v \vert^{2d-1}dudv \notag \\
&=& 2 C_{d} \int_{0}^{\infty}   \int_{0}^{u}  e^{\lambda_{k}(u+v) }  ( u-v )^{2d-1} dvdu\notag \\
&=& 2  C_{d} \int_{0}^{\infty}   e^{2 \lambda u}   \int_{0}^{u}  e^{-\lambda v }  v ^{2d-1} dvdu\notag \\
&=& 2  C_{d} \int_{0}^{\infty}  e^{-\lambda v }  v ^{2d-1}   \int_{v}^{\infty}  e^{2 \lambda u} du dv \notag \\
&=& -\dfrac{C_{d}}{\lambda} \int_{0}^{\infty}  e^{\lambda v } v^{2d-1} dv= \dfrac{C_{d}}{(-\lambda)^{2d+1} } \Gamma(2d).
\end{eqnarray}

\

\begin{remark}
For $c\neq 0$, by \eqref{eqEV2} we have that $E_{L}[(V^{d}(ct))^{2}]=  \dfrac{C_{d}}{(-\lambda)^{2d+1} } \Gamma(2d)$. Hence the processes $V^{d}$ is no self-similar.
 \end{remark}

\

\begin{lemma}\label{stationarity}
The process $V^{d}$ (for fixed $\omega_{\lambda} \in \Omega_{\lambda}$) is stationary, i.e.  for $b>0$ and $t_{1} < \ldots < t_{n} $, with $n\in \mathbb{N}$ 
$$(V^{d}(t_{1} +b), \ldots , V^{d}(t_{n} + b))  \stackrel{(d)}{=} (V^{d}(t_{1} ), \ldots , V^{d}(t_{n} )),$$
where $\stackrel{(d)}{=}$ means equality in the sense of finite dimensional distributions. 
\end{lemma}
\begin{proof}
For $b>0$, $u_{1}, \ldots , u_{n} $ and $-\infty < t_{1} < \ldots < t_{n} $, $n \in \mathbb{R}$, by the stationarity of the increments of $L^{d}$ we get:
\begin{eqnarray*}
\sum_{i=1}^{n} u_{i} V^{d}(t_{i} +b) &=& \sum_{i=1}^{n} u_{i} \int_{-\infty}^{t_{i} + b} e^{\lambda(t_{i} + b -u) }dL^{d}(u) \\
&\stackrel{(d)}{=}& \sum_{i=1}^{n} u_{i} \int_{-\infty}^{t_{i}} e^{\lambda (t_{i} -u) }dL^{d}(u) \\
&=&\sum_{i=1}^{n} u_{i} V^{d}(t_{i}). 
\end{eqnarray*}
\qed
\end{proof}  
Using the results from \cite{mar-stel, mar2} already used in the reference  \cite{mar-jam} to construct long memory process based on CARMA process driven by Lévy processes. We will provide a spectral representation of the process $V^{d}$  given by (\ref{v}). 

\

 Let $L=\{ L(t) \}_{t \in \mathbb{R}}$ be a two sided square integrable Lévy process with $E[L(1)]=0$ and 
$E[L(1)^{2}]=\Sigma_{L}$, then there exists a random orthogonal measure $\Phi_{L}$ with spectral measure $F_{L}$ such that  $E[\Phi_{L}(\Delta)]=0$ for any bounded Borel set $\Delta$, 
$$F_{L}(dt) = \dfrac{\Sigma_{L}}{2\pi} dt.$$
Also, the random measure $\Phi_{L}$ is uniquely determined by 

\begin{equation}
\label{edefphi}
  \Phi_{L}([a,b)) = \int_{\mathbb{R}} \dfrac{e^{-i\alpha a } - e^{-i\alpha b }}{2\pi i \alpha} L(d\alpha),  
\end{equation}
for all $-\infty <a<b<\infty$. Moreover,
$$L(t) = \int_{\mathbb{R}} \dfrac{e^{i\alpha t} -1}{i \alpha} \Phi_{L}(d\alpha), \quad t \in \mathbb{R}.$$
Hence, for any function $f \in L^{2}(\mathbb{C})$,
\begin{equation}
\int_{\mathbb{R}} f(\alpha) \Phi_{L}(d\alpha) =   \dfrac{1}{2\pi} \int_{\mathbb{R}} \int_{\mathbb{R}} e^{-i\alpha t} f(\alpha) d\alpha L(dt) =  \dfrac{1}{\sqrt{2\pi}} \int_{\mathbb{R}} \hat{f}(t) L(dt),
 \end{equation}
 \begin{equation}
 \int_{\mathbb{R}} \hat{f}(t) L(dt) = \int_{\mathbb{R}} \int_{\mathbb{R}} e^{i\alpha t} \hat{f}(t)  dt \Phi_{L}(d\alpha) = \sqrt{2\pi} \int_{\mathbb{R}} f(\alpha) \Phi_{L}(d\alpha), 
 \end{equation}
 where 
 \begin{equation*}
 \hat{f}(t) = \dfrac{1}{\sqrt{2\pi}} \int_{\mathbb{R}} e^{-i\alpha t} f(\alpha) d\alpha \quad \mbox{and} \quad  f(\alpha)  = \dfrac{1}{\sqrt{2\pi}} \int_{\mathbb{R}} e^{i\alpha t} \hat{f}(t) d\alpha.  
 \end{equation*}

 \
 
 \begin{lemma}
 Let $V^{d}$  be a fLpOU with random coefficient given by (\ref{v}). Then, the spectral representation of $V^{d}$ is 
 \begin{equation}\label{spectral}
 V^{d}(t) = \int_{\mathbb{R}} e^{it\alpha} \dfrac{1}{i\alpha - \lambda} \Phi_{M}(d\alpha), \quad t\in \mathbb{R},
 \end{equation}
 where 
 $$\Phi_{M}([a,b]) = \int_{\mathbb{R}} \dfrac{e^{-ias} - e^{-ibs}}{2\pi is} dL^{d}(s).$$
\end{lemma}
 \begin{proof}
 By Theorem 2.8 in \cite{mar2}, we know that 
 $$L^{d}(t) = \int_{\mathbb{R}} (e^{i\alpha t} -1) (i\alpha)^{-(d+1)} \Phi_{L}(d\alpha), \quad t\in \mathbb{R}.$$
 Furthermore, following the proof of Theorem 2.8,  Remark 2.9 and equality (2.31) in the same reference, we can obtain 
 \begin{eqnarray*}
 V^{d}(t) &=&  \int_{\mathbb{R}} e^{\lambda (t-u) } 1_{(-\infty , t)} (u)dL^{d}(u) \\
 &=& \int_{\mathbb{R}} \int_{\mathbb{R}} e^{i\alpha u} e^{\lambda (t-u) } 1_{(-\infty , t)} (u) du  \Phi_{M}(d\alpha) \\
 &=& \int_{\mathbb{R}} e^{it\alpha} \dfrac{1}{i\alpha - \lambda} \Phi_{M}(d\alpha).
 \end{eqnarray*} 
 Then, the result is achieved. \qed
  \end{proof}
 
 \begin{corollary}
 With almost  no major effort, we can obtain, by Remark 2.9 in \cite{mar2}, that 
 \begin{equation*}
  V^{d}(t)  =  \int_{\mathbb{R}} e^{it\alpha} \dfrac{1}{i\alpha - \lambda_{k}} (i\alpha)^{-d}  \Phi_{L}(d\alpha), \quad t\in \mathbb{R}.
 \end{equation*}
 \end{corollary}

 \section{Gamma mixed fractional Lévy Ornstein - Uhlenbeck process}\label{aou}

In this section, first,  we will study a process defined by the aggregation of independent fLOUp with random coefficient. Then we study its limit process,  which  we call Gamma mixed fractional Lévy Ornstein - Uhlenbeck process. Some properties of this limit process are given, namely,  we give the characteristic function of the finite dimensional distributions of the process, we determinate its covariance, stationarity and long memory property; and finally,  we analyze the asymptotic behavior with respect to the parameter $\alpha$.  

\subsection{Aggregated fractional Lévy Ornstein-Uhlenbeck  process}
Consider a sequence of fLOUp with random coefficient $V_{k}^{d} = (V_{k}^{d}(t) , t \in \mathbb{R})$, $k \geq 1$ given by
\begin{equation}\label{vk}
V_k^{d}(t) = \int_{-\infty}^{t} e^{\lambda_k (t-u) }dL^{d}(u), \quad t \in \mathbb{R}, \quad k \geq 1, 
 \end{equation}
where $L^{d}$ is a fLp with $d \in (0,1/2)$  defined on $(\Omega_{L}, {\cal F}_{L}, \mathbb{P}_{L})$. The random variables $\lambda_{k}$ are assumed independent and identically distributed defined on the probability space $(\Omega_{\lambda}, {\cal F}_{\lambda}, \mathbb{P}_{\lambda})$, and for $k \geq 1$ we assume that $-\lambda_{k}$ follows a Gamma distribution with parameters $1-h$ and $\alpha$, i.e. $-\lambda_{k} \sim \Gamma(1-h, \alpha)$ with $h \in (0,1)$ and $\alpha >0$. Furthermore,  we also assume that the random variables $(\lambda_{k})_{k \geq 1}$ are independent from $L^{d}$.

\

The aggregated fractional Lévy Ornstein-Uhlenbeck  processs $Z_{m}^{d}=(Z_{m}^{d}(t), t  \in \mathbb{R} )$  is defined by 
\begin{equation}
\label{defaggreg}
Z_{m}^{d}(t) := \dfrac{1}{m} \sum_{k=1}^{m}  V_{k}^{d}(t), 
\end{equation}
for $m \geq 1$ and $d \in (0, 1/2)$. \\

\begin{lemma}\label{stationarity2}
For all $m\geq 1$, the aggregated fractional Lévy Ornstein-Uhlenbeck  processs $Z^{d}_m$  is stationary  and
$$E_{L}[(Z_{m}^{d}(t))^{2}]=\frac{C_{d}}{m^2} 2\Gamma(2d)\sum_{k=1}^m \sum_{j=1}^m \frac{1}{-(\lambda_{k} + \lambda_{j})(-\lambda_{k})^{2d}},$$ 
where $C_d$ is given by \eqref{eqcd}.
\end{lemma}

\begin{proof}
The similarity follows by Lemma \ref{stationarity}.

From Lemma \ref{lemma2} and the change of variables $\tilde{u} = t - u $ and $\tilde{v} = t - v$, we can see 

\begin{eqnarray*}
&& E_{L}[(Z_{m}^{d}(t))^{2}]\\
&=& \frac{C_{d}}{m^2} \int_{-\infty}^{t} \int_{-\infty}^{t} \sum_{k=1}^m \sum_{j=1}^m e^{\lambda_{k}(t-u) } e^{\lambda_{j}(t-v) } \vert u-v \vert^{2d-1}dv du  \notag \\
&=& \frac{C_{d}}{m^2} \sum_{k=1}^m \sum_{j=1}^m \int_{0}^{\infty} \int_{0}^{\infty} e^{\lambda_{k}u+ \lambda_{j}v }  \vert u-v \vert^{2d-1}  dv du\notag \\
&=&  \frac{C_{d}}{m^2} \sum_{k=1}^m \sum_{j=1}^m \left( \int_{0}^{\infty} \int_{0}^{u}  e^{\lambda_{k}u+\lambda_{j}v }  ( u-v )^{2d-1}dvdu \right. + \left. \int_{0}^{\infty} \int_{u}^{\infty}  e^{\lambda_{k}u+\lambda_{j}v }  ( v-u )^{2d-1}dvdu \right)\notag \\
&=&  \frac{C_{d}}{m^2} \sum_{k=1}^m \sum_{j=1}^m \left( \int_{0}^{\infty}  e^{ (\lambda_{k} + \lambda_{j}) u}   \int_{0}^{u}  e^{-\lambda_{j}v }  v ^{2d-1} dv du \right. 
 + \left. \int_{0}^{\infty}  e^{ (\lambda_{k} + \lambda_{j}) u}   \int_{0}^{\infty}  e^{\lambda_{j}v }  v ^{2d-1} dv du \right)\notag \\
&=& \frac{C_{d}}{m^2} \sum_{k=1}^m \sum_{j=1}^m  \left(\int_{0}^{\infty}  e^{\lambda_{j}v }  v ^{2d-1}\int_v^\infty e^{(\lambda_k+\lambda_j)u}dudv \right. 
 + \left.  \int_{0}^{\infty}  e^{\lambda_{j}v }  v ^{2d-1}\int_0^\infty e^{(\lambda_k+\lambda_j)u}dudv \right) \notag \\
&=& \frac{C_{d}}{m^2} \sum_{k=1}^m \sum_{j=1}^m \frac{1}{-(\lambda_{k} + \lambda_{j})} \left(\int_{0}^{\infty}  e^{\lambda_{k}v }  v ^{2d-1}dv +   \int_{0}^{\infty}  e^{\lambda_{j} v}v^{2d-1} dv \right) \notag \\
&=& \frac{C_{d}}{m^2} \Gamma(2d)\sum_{k=1}^m \sum_{j=1}^m \frac{1}{-(\lambda_{k} + \lambda_{j})} \left(\dfrac{1}{(-\lambda_{k})^{2d} } + \dfrac{1}{(-\lambda_{j})^{2d} }\right)\\
&=& \frac{C_{d}}{m^2} 2\Gamma(2d)\sum_{k=1}^m \sum_{j=1}^m \frac{1}{-(\lambda_{k} + \lambda_{j})(-\lambda_{k})^{2d}}. 
\end{eqnarray*}
\qed
\end{proof}

\subsection{Limit of aggregated fractional Lévy Ornstein-Uhlenbeck  process}
By (\ref{v}) we can see that $Z_{m}^{d}$ can be written as 
  \begin{eqnarray}
\label{zm}
Z_{m}^{d}(t) &=&  \int_{-\infty}^{t}  \dfrac{1}{m} \sum_{k=1}^{m} e^{\lambda_{k}(t-u) }dL^{d}(u) \nonumber \\
&=& \int_{-\infty}^{t} f_{m}(t-u)dL^{d}(u). 
\end{eqnarray}
Moreover, by the law of large numbers we get   
 \begin{equation}
\dfrac{1}{m} \sum_{k=1}^{m} e^{\lambda_{k}(t-u) } \rightarrow  E_{\lambda}\left[ e^{\lambda_{1}(t-u)} \right]= \left(\dfrac{\alpha}{\alpha + t -u} \right)^{1-h},
\end{equation}
as $m \rightarrow \infty$, where the convergence is $P_{\lambda}$ almost surely.\\

Due to the previous result, a natural candidate to be the limit of the aggregated fractional Lévy Ornstein-Uhlenbeck  process is the process $Z^d=(Z^{d}(t), t\in \mathbb{R})$ given by  
\begin{eqnarray}
Z^{d}(t) &:=&  \int_{-\infty}^{t}  \left(\dfrac{\alpha}{\alpha + t -u} \right)^{1-h}    dL^{d}(u) \nonumber \\
 &=&  \int_{-\infty}^{t} g(t-u)  dL^{d}(u), \label{z}
\end{eqnarray}
where 
$g(t)= \left(\dfrac{\alpha}{\alpha + t } \right)^{1-h} $.\\

We can see that the process $Z^d$ belongs to $ L^{2}(\Omega_{L})$ if $0<h<1/2- d$. Actually, by Lemma \ref{lemma2} and applying consecutively the change of variables $\tilde{u} = t-u$,   $\tilde{v} = t-v$, $x=u/\alpha$, $y=v/\alpha$, $r=1/(1+x)$, $s=1/(1+y)$, and $\hat{u}=r/s$, we obtain
\begin{eqnarray*}
E_{L}\left[ (Z^{d}(t))^2  \right] &=& C_{d} \int_{-\infty}^{t} \int_{-\infty}^{t} g(t-u) g(t-v) \vert u-v \vert^{2d-1} dvdu \\
& =& C_{d} \int_{0}^{\infty} \int_{0}^{\infty} g(u) g(v) \vert u-v \vert^{2d-1} dvdu\\
 &=& C_{d} \int_{0}^{\infty} \int_{0}^{\infty} \left(1 + \dfrac{u}{\alpha} \right)^{h-1}  \left(1 + \dfrac{v}{\alpha} \right)^{h-1} \vert u-v \vert^{2d-1} dvdu \\
&=& C_{d} \alpha^{2d+1} \int_{0}^{\infty} \int_{0}^{\infty} \left(1 +x \right)^{h-1}  \left(1 + y \right)^{h-1} \vert x-y \vert^{2d-1} dydx \\
&=& C_{d} \alpha^{2d+1} \int_{0}^{1} \int_{0}^{1} r^{1-h}  s^{1-h}  (rs)^{-2}\left\vert \dfrac{1}{r} - \dfrac{1}{s}\right\vert^{2d-1} drds \\
&=& 2C_{d} \alpha^{2d+1} \int_{0}^{1} s^{-1-h} \int_{0}^{s} r^{-h-2d} \left( 1 - \dfrac{r}{s}\right)^{2d-1} drds \\
&=& 2  C_{d} \alpha^{2d+1} \int_{0}^{1} s^{-2(h+d)} \int_{0}^{1} \hat{u}^{-h-2d} (1-\hat{u})^{2d-1} d\hat{u}  ds \\
&=& \dfrac{ 2  C_{d} \alpha^{2d+1} }{1-2(h+d)}B(1-h-2d,2d) = C_{\alpha, h,d}.
\end{eqnarray*}
Clearly, from the last line emerges the condition $0<h<1/2- d$.\\

Now, we present the main result of this section related to the limit of the aggregated fractional Lévy Ornstein-Uhlenbeck  process. This result is analogous to that obtained in  Theorem 3 in \cite{es1} for the aggregated fractional Ornstein-Uhlenbeck processes.

\begin{theorem}
Let $Z_{m}^{d}$ and $Z^{d}$ be defined by (\ref{defaggreg}) and (\ref{z}), respectively. Assume that $0<h<1/2- d$. Then $P_{\lambda}$-a.s., for every $t \in \mathbb{R}$ 
\begin{equation}
\label{eqteo1}
Z_{m}^{d}(t) \longrightarrow Z^{d}(t), \qquad \mbox{in} \quad  L^{2}(\Omega_{L}) 
\end{equation} 
and for $a,b \in \mathbb{R}$ 
\begin{equation}
\label{eqteo2}
Z_{m}^{d}(t) \longrightarrow Z^{d}(t), \qquad \mbox{weakly \ in} \ \ C[a, b] \ \ \mbox{under} \quad P_{L}, 
\end{equation}
as $m \rightarrow \infty$.  
\end{theorem}
\begin{proof}
By  Lemma \ref{lemma2}, (\ref{zm}) and (\ref{z}), we have
\begin{eqnarray*}
E_{L}[(Z^{d}_{m}(t) - Z^{d}(t))^{2}] &=& E_{L}\left[ \left( \int_{-\infty}^{t} \left[ \dfrac{1}{m} \sum_{k=1}^{m} e^{\lambda_{k}(t-u)} - \left( \dfrac{\alpha}{\alpha + t-u} \right)^{h}     \right] dL^{d}(u)\right)^{2}\right] \\
 &=& E_{L}\left[ \left( \int_{-\infty}^{t} \left[ f_{m}(t-u) -  g(t-u)   \right] dL^{d}(u)\right)^{2}\right]\\
 &=& C_{d} \int_{0}^{\infty} \int_{0}^{\infty} (f_{m}(u) -  g(u))(f_{m}(v) -  g(v))) \vert u -v \vert^{2d-1}  dudv.
\end{eqnarray*}
Now, following the lines of the proof of Theorem 3 (first part) in \cite{es1} and the fact that $0<h<1/2- d$, we can obtain that $P_{\lambda}$-a.s. 
\begin{equation*}
\lim\limits_{m \rightarrow \infty} E_{L}[(Z^{d}_{m}(t) - Z^{d}(t))^{2}]  = 0.
\end{equation*}

Hence we have the $P_{\lambda}$-a.s. convergence of the sequence $(Z^d_m)_{m\geq 1}$ in $L^2(\Omega_L)$, which in turn implies the $P_{\lambda}$-a.s. convergence of the finite dimensional distributions.

In order to prove the weak convergence, we only need to prove  the tightness.  Due Theorem 12.3 in \cite{billi},  it is sufficient to show that 
$E_{L}[(Z_{m}^{d}(t) - Z_{m}^{d}(s))^{2} ] \leq C(t-s)^{1+\rho}$, for $T>0$ fix and $0 \leq s <t \leq T$  where $\rho >0$ and C may depend uppon parameters.  

From Lemma \ref{stationarity2} and \eqref{zm}
\begin{eqnarray*}
Z_{m}^{d}(t) - Z_{m}^{d}(s) &\stackrel{(d)}{=}& Z_{m}^{d}(t-s) -Z_{m}^{d}(0)\\
 &\stackrel{(d)}{=}& \dfrac{1}{m} \sum_{k=1}^{m} (e^{\lambda_{k}(t-s)}-1) \int_{-\infty}^{0} e^{-u\lambda_{k}} dL^{d}(u) + \int_{0}^{t-s}  \dfrac{1}{m} \sum_{k=1}^{m} e^{\lambda_{k}(t-s-u)} dL^{d}(u) \\
&=& I_{1} + I_{2}.
\end{eqnarray*}
Clearly, this implies 
\begin{equation*}
E_{L}[(Z_{m}^{d}(t) - Z_{m}^{d}(s))^{2} ] \leq 2 E_{L}[I_{1}^{2}] +  2 E_{L}[I_{2}^{2}].
\end{equation*}
We will estimate every term separately. In fact, by Lemma \ref{lemma2} we have
\begin{eqnarray*}
E_{L}[I_{1}^{2}] & = & \dfrac{1}{m^2}\sum_{k_{1}, k_{2} =1}^{m}   (e^{\lambda_{k_{1}}(t-s)}-1) (e^{\lambda_{k_{2}}(t-s)}-1) E_{L} \left[  \int_{-\infty}^{0} e^{-u\lambda_{k_{1}}} dL^{d}(u)  \int_{-\infty}^{0} e^{-v\lambda_{k_{1}}} dL^{d}(v)   \right] \\
&=& \dfrac{1}{m^2}\sum_{k_{1}, k_{2} =1}^{m}   (e^{\lambda_{k_{1}}(t-s)}-1) (e^{\lambda_{k_{2}}(t-s)}-1)   \int_{-\infty}^{0}  \int_{-\infty}^{0}  e^{-u\lambda_{k_{1}}}  e^{-v\lambda_{k_{1}}} \vert u-v\vert^{2d-1} dudv. 
\end{eqnarray*}  
By similar arguments to those  of \cite{dou}, the law of large number and the fact that $0<h<1/2-d$
\begin{equation}\label{i1}
 E_{L}[I_{1}^{2}] \leq C_{d,\lambda} (t-s)^{2}.
\end{equation}

With respect to $I_{2}$,  by Lemma \ref{lemma2} and making the change of variable $z=u-v$, we get 
\begin{eqnarray*}
E_{L}[I_{2}^{2}] & = & E_{L} \left[  \int_{0}^{t-s}  \dfrac{1}{m} \sum_{k=1}^{m} e^{\lambda_{k}(t-s-u)} dL^{d}(u)  \cdot  \int_{0}^{t-s}  \dfrac{1}{m} \sum_{k=1}^{m} e^{\lambda_{k}(t-s-v)} dL^{d}(v)    \right]\\
 & = &  C_{d} \dfrac{1}{m^{2}} \sum_{k_{1},k_{2}=1}^{m} \int_{0}^{t-s}\int_{0}^{t-s} e^{\lambda_{k_{1}}(t-s-u)}  e^{\lambda_{k_{2}}(t-s-v)}     \vert u-v\vert^{2d-1} dudv \\
& = &  C_{d} \dfrac{1}{m^{2}} \sum_{k_{1},k_{2}=1}^{m}  e^{ \lambda_{k_{1}}(t-s)} e^{ \lambda_{k_{2}}(t-s)} \int_{0}^{t-s}\int_{0}^{t-s} e^{-\lambda_{k_{1}}u - \lambda_{k_{2}}v}   \vert u-v\vert^{2d-1} dvdu \\
& = &  2 C_{d} \dfrac{1}{m^{2}} \sum_{k_{1},k_{2}=1}^{m} e^{ \lambda_{k_{1}}(t-s)} e^{ \lambda_{k_{2}}(t-s)} \int_{0}^{t-s} e^{-\lambda_{k_{1}}u}   \int_{0}^{u} e^{-\lambda_{k_{2}}v}   ( u-v)^{2d-1} dvdu\\
& = &   2 C_{d} \dfrac{1}{m^{2}} \sum_{k_{1},k_{2}=1}^{m}  e^{ \lambda_{k_{1}}(t-s)} e^{ \lambda_{k_{2}}(t-s)} \int_{0}^{t-s}e^{-u (\lambda_{k_{1}} + \lambda_{k_{2}} )}   \int_{0}^{u} e^{\lambda_{k_{2}}z}   z^{2d-1} dvdu \\
& \leq & 2 C_{d} \dfrac{1}{m^{2}} \sum_{k_{1},k_{2}=1}^{m}    e^{ \lambda_{k_{1}}(t-s)} e^{ \lambda_{k_{2}}(t-s)}    \int_{0}^{t-s}e^{-\lambda_{k_{1}}u}      u^{2d} du. \\
& \leq & 2 C_{d} (t-s)^{1+2d} \dfrac{1}{m} \sum_{k_{2}=1}^{m}   e^{ \lambda_{k_{2}}(t-s)}. 
\end{eqnarray*}
 Then,  using the fact that $-\lambda_{k} \sim \Gamma(1-h, \alpha)$ with $h \in (0,1)$ and $\alpha >0$, and $t>s$. We obtain
\begin{equation}\label{i2}
E_{L}[I_{2}^{2}]  \leq  2 C_{d} (t-s)^{1+2d}.
\end{equation}
Finally, inequalities (\ref{i1}) and (\ref{i2}) imply

\begin{equation*}
E_{L}[(Z_{m}^{d}(t) - Z_{m}^{d}(s))^{2} ] \leq C_{T, d, \lambda, \alpha, h} (t-s)^{1+2d}.
\end{equation*}
Since $d \in (0,1/2)$ the result is achieved. 
\qed
\end{proof}
\subsection{Properties of $\mathbf{Z^d}$}

Now,  we study some  properties of the limit process $Z^{d}$.  

\begin{prop}\label{long-memory}
Let $Z^{d}$ be given by (\ref{z}) with $h \in (0, 1/2 -d)$. Then, the characteristic function  of  $Z^{d}(t_{1}), Z^{d}(t_{2}), \ldots , Z^{d}(t_{m})$ with $t_{1} < t_{2}< \ldots < t_{m} $ is given by 
\begin{equation*}
E_{L}\left[ \exp\left(  i \sum_{j=1}^{m} \theta_{j} Z^{d}(t_{j}) \right) \right] =  \exp\left(  \int_{\mathbb{R}} 
\psi\left[      \sum_{j=1}^{m} \theta_{j} \tilde{f}_{t_{j},h, \alpha} (s) \right]
 ds \right) ,
\end{equation*}
where 
$$ \tilde{f}_{t_{j},h,\alpha}(s) =  d  \int_{-\infty}^{t_{j}} \left(\dfrac{\alpha}{\alpha + t_{j} -v} \right)^{1-h} ( v -s )_{+}^{d-1}  dv.$$
\end{prop}
\begin{proof}
The result follows by Lemma \ref{lemma3} in Section 2 (also see Proposition 3.3 in \cite{klau}) and the fact that $Z^{d}$ belongs to $L^{2}(\Omega_{L})$ for $h \in (0, 1/2 -d)$.  
\qed
\end{proof}

\begin{prop}
Let $Z^{d}$ be given by (\ref{z}) with $h \in (0, 1/2 -d)$. Then, $Z^{d}$ is a stationary process with long memory.
\end{prop}
\begin{proof}
By (\ref{z}), the stationarity of the increments of $L^{d}$ and making the change of variable $\tilde{u} = u-b$, we can get for every $b>0$

\begin{eqnarray*}
Z^{d}(t+b) &=&  \int_{-\infty}^{t+b}  \left(\dfrac{\alpha}{\alpha + t+b -u} \right)^{1-h}    dL^{d}(u)  \\
&\stackrel{(d)}{=}&  \int_{-\infty}^{t}  \left(\dfrac{\alpha}{\alpha +t -u} \right)^{1-h}    dL^{d}(u) = Z^{d}(t).  
\end{eqnarray*}
With respect to the long memory property we will study the non summability of the covariance. Since $Z^{d}$ is a stionary process, we have 
\begin{eqnarray*}
E_{L}[Z^{d}(a)Z^{d}(t+a)] &=&  E_{L}[Z^{d}(0)Z^{d}(t)].
\end{eqnarray*}
Then, Lemma \ref{lemma2} implies
\begin{eqnarray*}
 E_{L}[Z^{d}(0)Z^{d}(t)] &=& C_{d} \int_{-\infty}^{0}\int_{-\infty}^{t} \left(\dfrac{\alpha}{\alpha + t -u} \right)^{1-h}  \left(\dfrac{\alpha}{\alpha  -v} \right)^{1-h} \vert u-v\vert^{2d-1} dudv\\ 
 &=& C_{d} \alpha^{2-2h} \int_{-\infty}^{0}\int_{-\infty}^{t} \left(\alpha + t -u\right)^{h-1}  \left( \alpha  -v \right)^{h-1} \vert u-v\vert^{2d-1} dudv\\
&=& C_{d} \alpha^{2-2h} t^{2h+2d-1}  \int_{-\infty}^{0}\int_{-\infty}^{1} \left( \dfrac{\alpha}{t} +1 -y\right)^{h-1}  \left(  \dfrac{\alpha}{t} -x \right)^{h-1} \vert x-y\vert^{2d-1} dxdy\\  
& \sim & C_{d} \alpha^{2-2h} t^{2h+2d-1}  \int_{-\infty}^{0}\int_{-\infty}^{1} \left( 1 -y\right)^{h-1}  \left(-x \right)^{h-1} \vert x-y\vert^{2d-1} dxdy\\  
&=& C_{d,\alpha,h}  t^{2h+2d-1}.
\end{eqnarray*}
Then, the long memory property is obtain by just noticing that $h+d>0$.
\qed
\end{proof}
\begin{remark}
Even if $d=0$, the long memory property is satisfied if $h>0$. \end{remark}
\begin{remark}
We can see that the process $Z^{d}$ has the following property 
$$ \dfrac{\rho(t)}{\rho(0)} \sim C_{\alpha, h , d}  t^{2h+2d-1}, \mbox{with} \quad  \rho(t) =E_{L}[Z^{d}(0)Z^{d}(t)],$$
i.e., the process is almost self-similar (see \cite{igloi} page 13 for details).  
\end{remark}

With respect to the behavior of the process $Z^{d}$ with respect to the parameter $\alpha$ we have the following results 
\begin{prop}
Let $t \geq 0$ and $d \in (0,1/2)$, then the random variable 
\begin{equation*}
Z^{d}(t) - Z^{d}(0)
\end{equation*}  
converges en $L^{2}(\Omega_{L})$ as $\alpha \rightarrow \infty$ to the random variable $L^{d}(t)$.
\end{prop}
\begin{proof}
By (\ref{z}) we obtain 
\begin{eqnarray*}
Z^{d}(t) - Z^{d}(0) &=& \int_{-\infty}^{t}  \left(\dfrac{\alpha}{\alpha + t -u} \right)^{1-h}    dL^{d}(u) -\int_{-\infty}^{0}  \left(\dfrac{\alpha}{\alpha  -u} \right)^{1-h}    dL^{d}(u)  \nonumber \\
&=&\int_{\mathbb{R}}\left[ \left(\dfrac{\alpha}{\alpha + t -u} \right)^{1-h} 1_{(-\infty,t)}(u)  -  \left(\dfrac{\alpha}{\alpha  -u} \right)^{1-h} 1_{(-\infty,0)}(u)   \right] dL^{d}(u) \\
&=& \int_{\mathbb{R}} f_{t,d}(u) dL^{d}(u). 
\end{eqnarray*}
Clearly, $f_{t,d}$ converges to $1_{(0,t)}$ as $\alpha \rightarrow \infty$. Therefore, our candidate as a limit will  be  
$L^{d}(t)$. In fact, by Lemma \ref{lemma2}
\begin{eqnarray*}
&&E_{L}[(Z^{d}(t) - Z^{d}(0) - L^{d}(t))^{2} ] \\
&=& C_{d} \int_{\mathbb{R}} \int_{\mathbb{R}}   \left[ \left(\dfrac{\alpha}{\alpha + t -u} \right)^{1-h} 1_{(-\infty,t)}(u)  -  \left(\dfrac{\alpha}{\alpha  -u} \right)^{1-h} 1_{(-\infty,0)}(u) - 1_{(0,t)}(u)    \right] \\
& \times & \left[ \left(\dfrac{\alpha}{\alpha + t -v} \right)^{1-h} 1_{(-\infty,t)}(v)  -  \left(\dfrac{\alpha}{\alpha  -v} \right)^{1-h} 1_{(-\infty,0)}(v) - 1_{(0,t)}(v)    \right] \vert u-v\vert^{2d-1} du dv. 
\end{eqnarray*}
Finally, the result is obtained by means of the dominated convergence theorem. \qed
\end{proof}
\vspace{0.5cm}
If  $\alpha \rightarrow 0$, then we get the following result. 
\begin{prop}
Let $t \geq 0$ and let us define $\tilde{Z}^{d}(t)$ by 
\begin{equation*}
\tilde{Z}^{d}(t) = \alpha^{h-1} \int_{0}^{t} Z^{d}(s)ds,
\end{equation*} 
then, as $\alpha \rightarrow 0$ the random variable  $\tilde{Z}^{d}(d)$ converges in $L^{2}(\Omega_{L})$ to the random variable $Y^{d}$ given by 
\begin{equation}\label{y}
Y^{d}(t) := \dfrac{1}{h} \int_{\mathbb{R}} [(t-u)^{h}_{+} - (-u)_{+}^{h}]dL^{d}(u)
\end{equation}
\end{prop}
\begin{proof}
By (\ref{z}), we get 
\begin{eqnarray*}
\tilde{Z}^{d}(t) &=& \alpha^{h-1} \int_{0}^{t}  \int_{-\infty}^{s}  \left(\dfrac{\alpha}{\alpha + s -u} \right)^{1-h}    dL^{d} (u) ds \\
&=&  \int_{0}^{t}  \int_{-\infty}^{s}  \left( \alpha + s -u \right)^{h-1}    dL^{d} (u) ds \\
&=&   \int_{-\infty}^{t} \int_{v \vee 0}^{t}  \left( \alpha + s -v \right)^{h-1}    ds  dL^{d} (v) \\
&=& \dfrac{1}{h}  \int_{-\infty}^{t}  \left[  \left( \alpha + t -v \right)^{h}  -\left( \alpha + (v \vee 0) -v \right)^{h}\right]dL^{d} (v)\\
&=& \dfrac{1}{h}  \int_{-\infty}^{t}  r_{\alpha, t,h}(v)dL^{d} (v).
\end{eqnarray*}
We can see that  $r_{\alpha, t,h}(v)$ converges to $(t-v)^{h}_{+} - (-v)_{+}^{h}$ for every $v$ as $\alpha \rightarrow 0$. Therefore,  by similar arguments to the previous proposition the result is obtained. \qed
\end{proof}

\begin{prop}\label{sec-y}
Let $Y^{d}=(Y^{d}(t))_{t\geq0}$ with $Y^{d}(t)$ be given by (\ref{y}), then $Y^{d}$  is a stationary process with 
$$E_{L}[(Y^{d}(t))^{2}] = t^{2h+2d+1} E_{L}[(Y^{d}(1))^{2}].$$  
\end{prop}
\begin{proof}
By (\ref{y}) and taking $b >0$
\begin{eqnarray*}
Y^{d}(t + b) - Y^{d}(b)   &=&  \dfrac{1}{h} \int_{\mathbb{R}} [(t + b -u)^{h}_{+} - (b -u)^{h}_{+}]dL^{d}(u) \\
&\stackrel{(d)}{=}&  \dfrac{1}{h} \int_{\mathbb{R}} [(t  -v)^{h}_{+} - ( -v)^{h}_{+}]dL^{d}(v) =  Y^{d}(t),
\end{eqnarray*}
here we have used the change of variable $v = u-b$ and the fact that $L^{d}$ is a stationary increment process. With respect to the second part of the statement, we have that 
\begin{eqnarray*}
E_{L}[(Y^{d}(t))^{2}] &=& \dfrac{1}{h^{2}} C_{d} \int_{\mathbb{R}} \int_{\mathbb{R}} ((t-u)^{h}_{+} - (-u)_{+}^{h})((t-v)^{h}_{+} - (-v)_{+}^{h}) \vert u-v \vert^{2d-1} dvdu\\
&=& t^{2h+2d+1}\dfrac{1}{h^{2}} C_{d}   \int_{\mathbb{R}} \int_{\mathbb{R}} ((1-u)^{h}_{+} - (-u)_{+}^{h})((1-v)^{h}_{+} - (-v)_{+}^{h}) \vert u-v \vert^{2d-1} dvdu\\
&=& t^{2h+2d+1} E_{L}[(Y^{d}(1))^{2}]. 
\end{eqnarray*}
\qed
\end{proof}

\begin{remark}
Let us note that we can write 
$$\tilde{Y}^{d}(t)  = \int_{\mathbb{R}} m_{t}(u) dL^{d}(u),$$
where $m_{t}(u) = \dfrac{1}{h}  [(t-u)^{h}_{+} - (-u)_{+}^{h}]$. Then, Proposition \ref{sec-y} and Theorem 3.1 in \cite{klau} imply
\begin{equation*}
Y^{d}(t) = \int_{\mathbb{R}} m_{t}(u)dL^{d}(u) =  \int_{\mathbb{R}} (I^{g}_{-m})(u) dL(s),
\end{equation*}
where (see Section 3 in the same reference for details) 
$$(I^{g}_{-m})(u) := \int_{u}^{\infty} m_{t}(v)g'(v-u)dv= \int_{\mathbb{R}} m_{t}(v)g'(v-u) dv $$
and $g(u)=(u)^{d}_{+}$. 

This implies that $Y^{d}$ can be seen as a type of Generalized fractional Lévy process (see \cite{klau} for details). 
\end{remark}

\section{Simulations}

In this section we are interested in obtain some simulations related to the process $Z^{d}$.
First,  let us recall how we can simulate a fractional Lévy process. 

In order to simulate sample paths from $L^{d}$ we use a Riemann-Stieltjes approximation, that is, we approximate $L^{d}$ in the following way  (see \cite{mar} for details)  
\begin{eqnarray*}
L^{d}(t) & \approx &  \frac{1}{\Gamma(d+1)} \left( \sum_{k=-n^2}^{0} \left[ \left( t - \frac{k}{n} \right)^{d} -  \left( - \frac{k}{n} \right)^{d} \right] \left( L_{(k+1)/n} - L_{k/n}  \right) \right. \\
&& + \left. \sum_{k=0}^{\lfloor nt \rfloor} \left( t - \frac{k}{n} \right)^{d}  \left( L_{(k+1)/n} - L_{k/n}  \right) \right). 
\end{eqnarray*}

\begin{figure}[H]
    \centering
             \includegraphics[width=7cm]{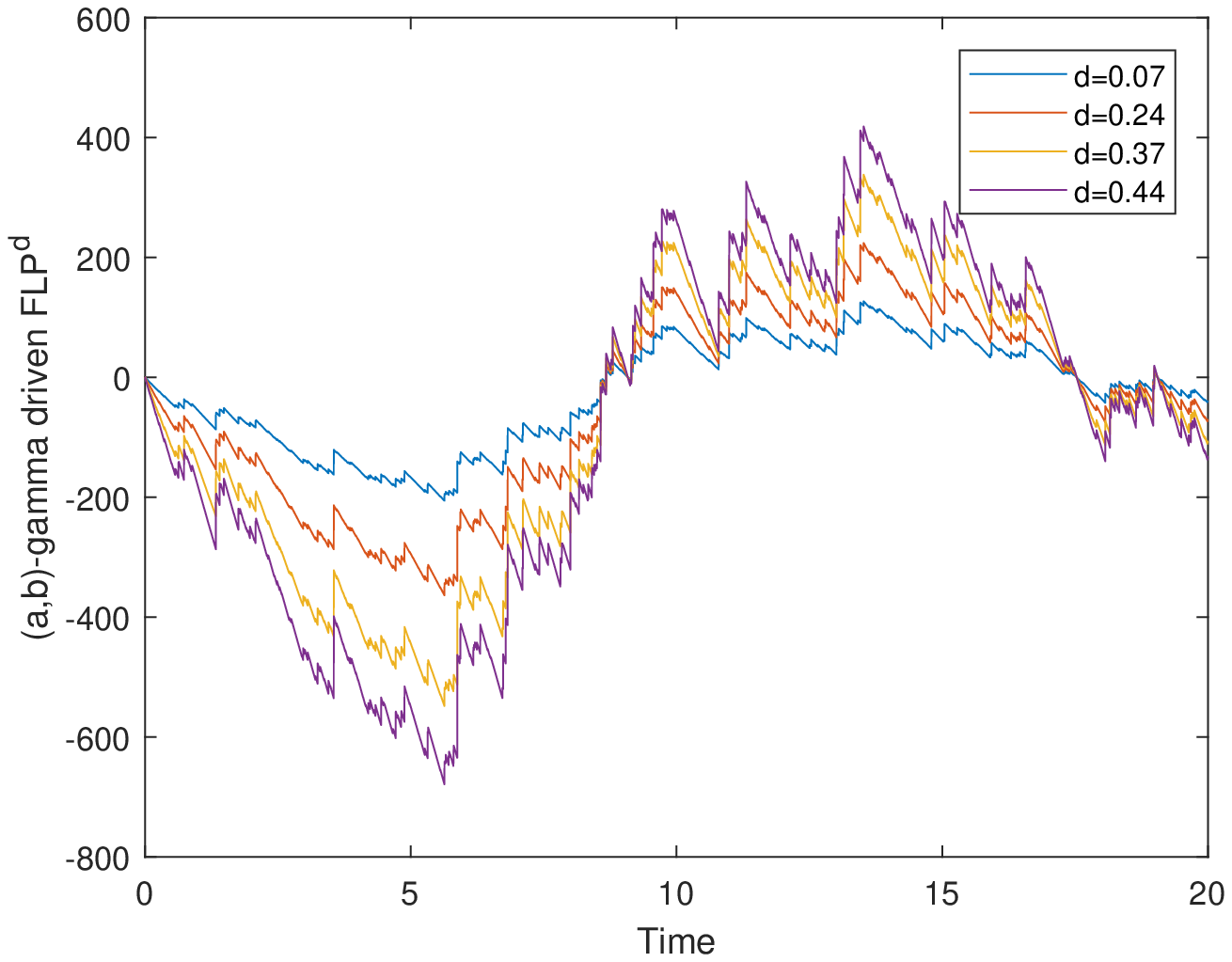}
              \includegraphics[width=7cm]{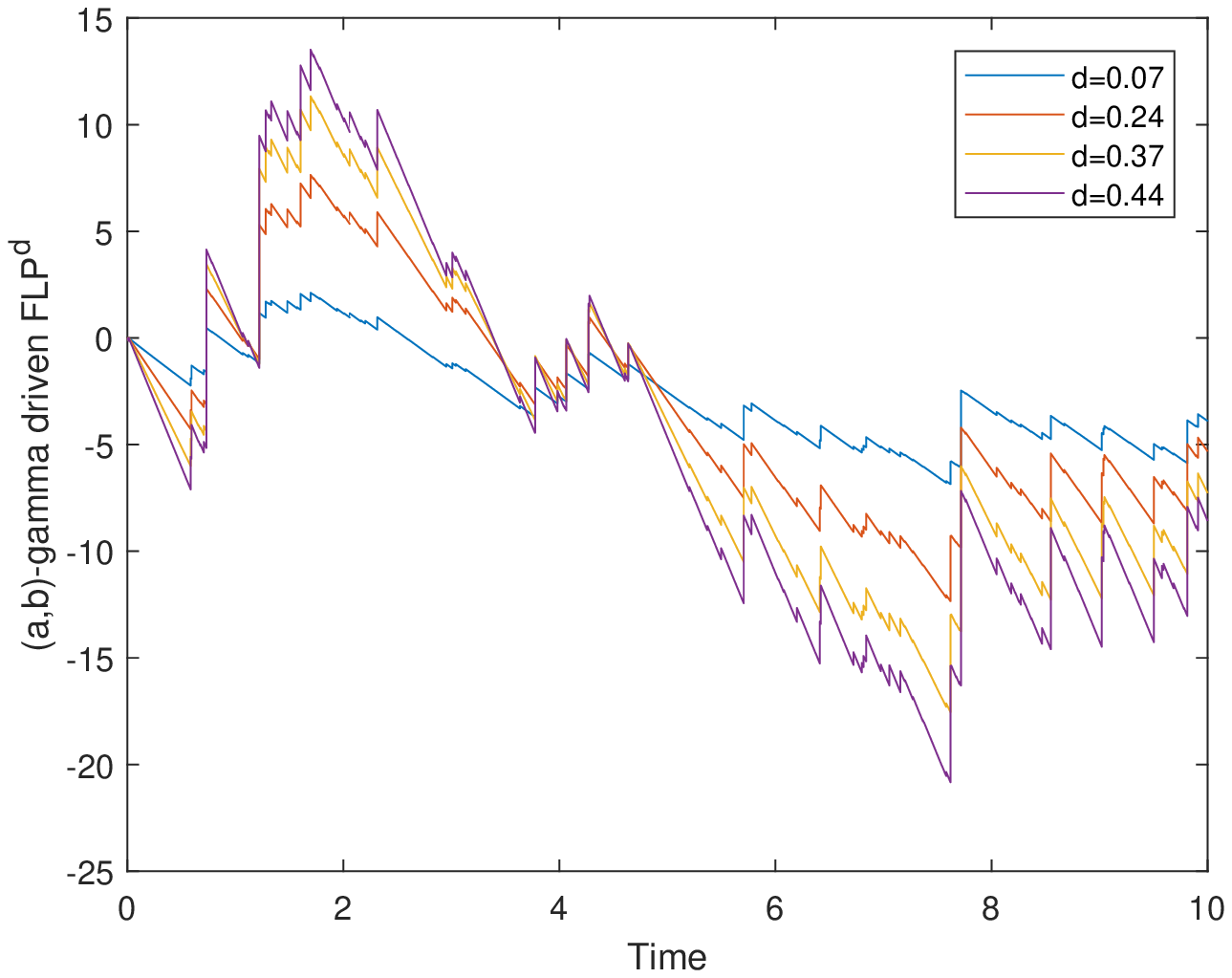}
      \caption{Sample paths of a fractional Lévy process for different values of $d$. The approximation is made using Riemann-Stieltjes approximation with driving Lévy process as a stationary Gamma process with $a=5$ and $b = 15$ ($a=1$ and $b = 2$).}
    \label{paths1}
\end{figure}

\begin{remark}
An optimal form of simulating $L^{d}$ is shown in \cite{taqqu}. Here, we have used  the  Riemann-Stieltjes approximation. This approximation can also be optimal if we take $a_{n} = n^{2-d/1-d}$. An advantage of this procedure is that simulate increments of a Lévy process is relatively easy (see \cite{cont} for details about the simulation of increments of a Lévy process).
\end{remark}

To simulate the process $Z^{d}$ we will use the following Riemann type approximation 
\begin{equation*}
Z^{d,(n)}(t)  = \alpha^{1-h} \sum_{k=-a_{n}}^{\lfloor nt \rfloor } \left(\alpha + t- \dfrac{k}{n} \right)^{h-1} \left(L^{d} \left( \dfrac{k+1}{n} \right) - L^{d} \left(\dfrac{k}{n} \right) \right), \quad t \in \mathbb{R},
\end{equation*}
as in  \cite{mar} we take $a_{n} = n^{2}$. 

\begin{figure}[H]
    \centering
             \includegraphics[width=7cm]{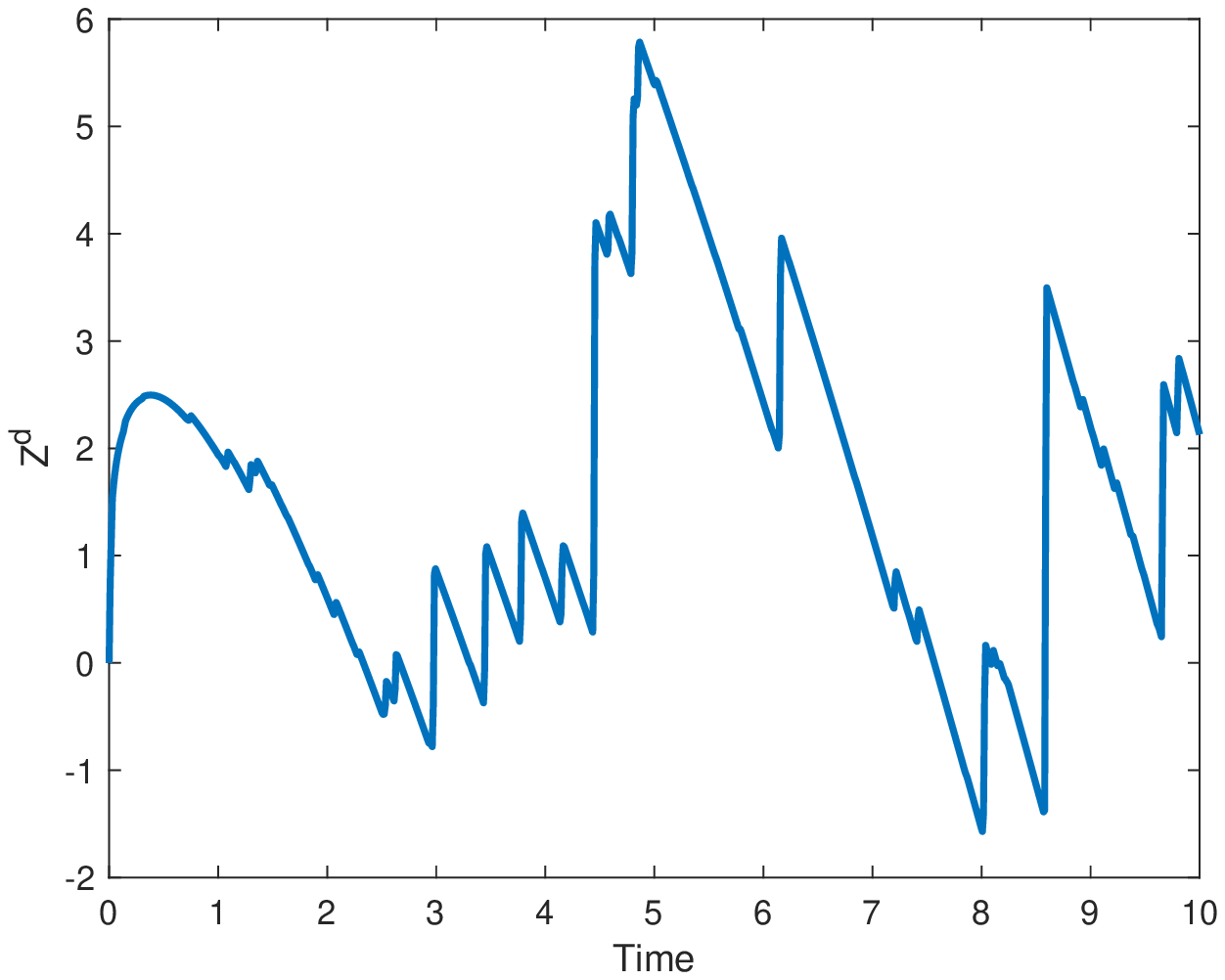}
              \includegraphics[width=7cm]{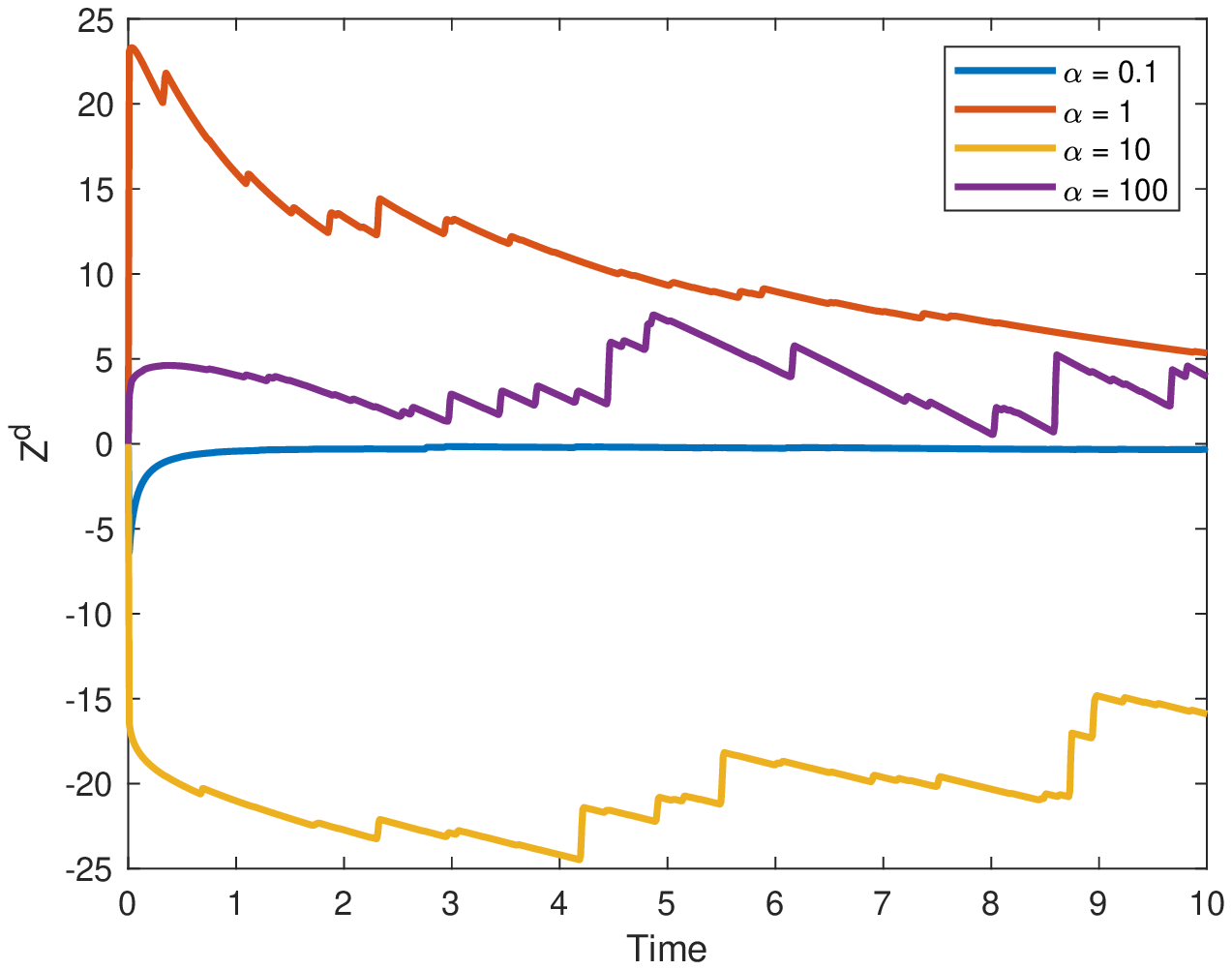}
      \caption{Sample paths $Z^d$ of the limit process for different $\alpha$ and $h= 0.12$ $(a=1, b=2)$.}
    \label{paths2}
\end{figure}

\begin{figure}[H]
    \centering
             \includegraphics[width=7cm]{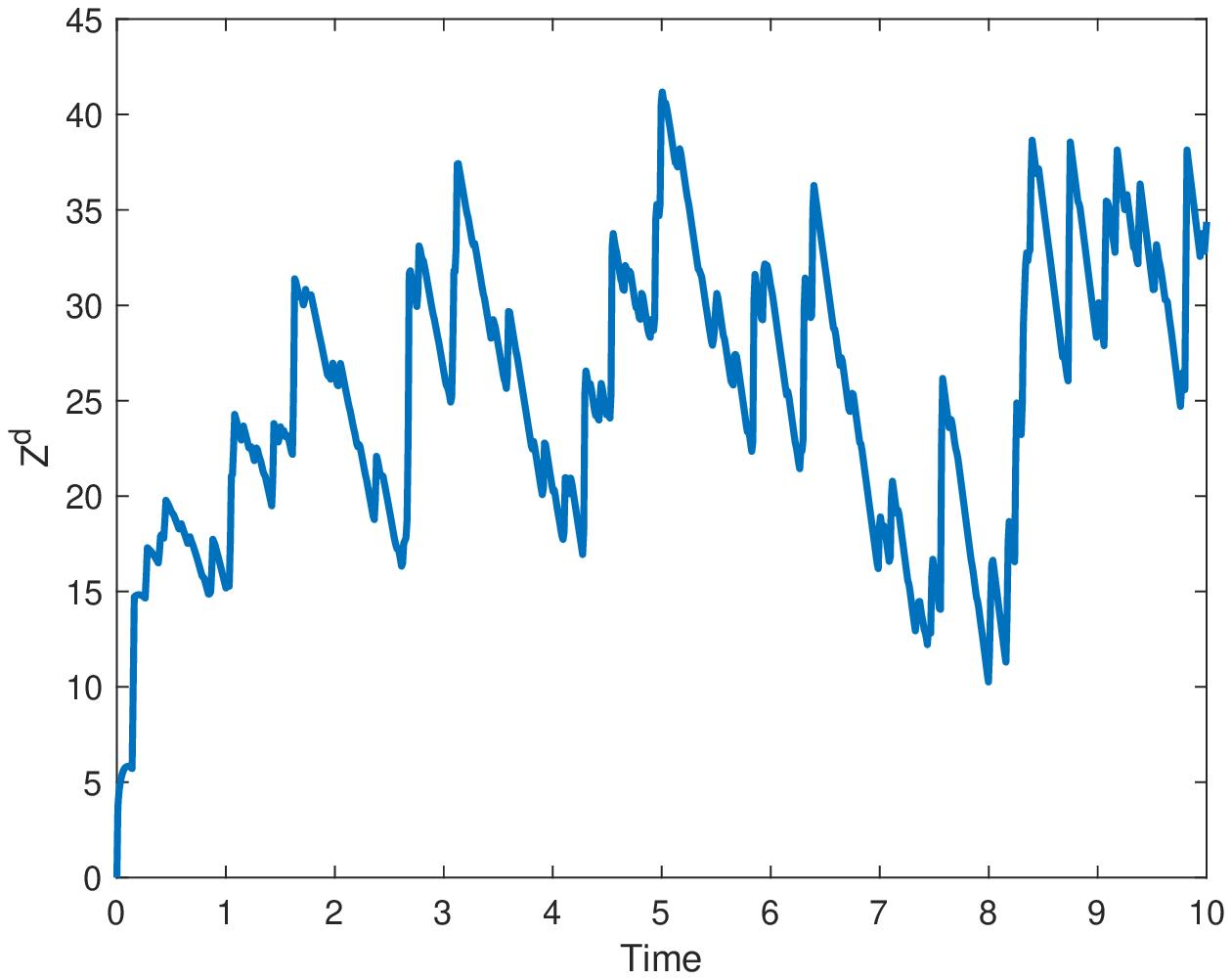}
              \includegraphics[width=7cm]{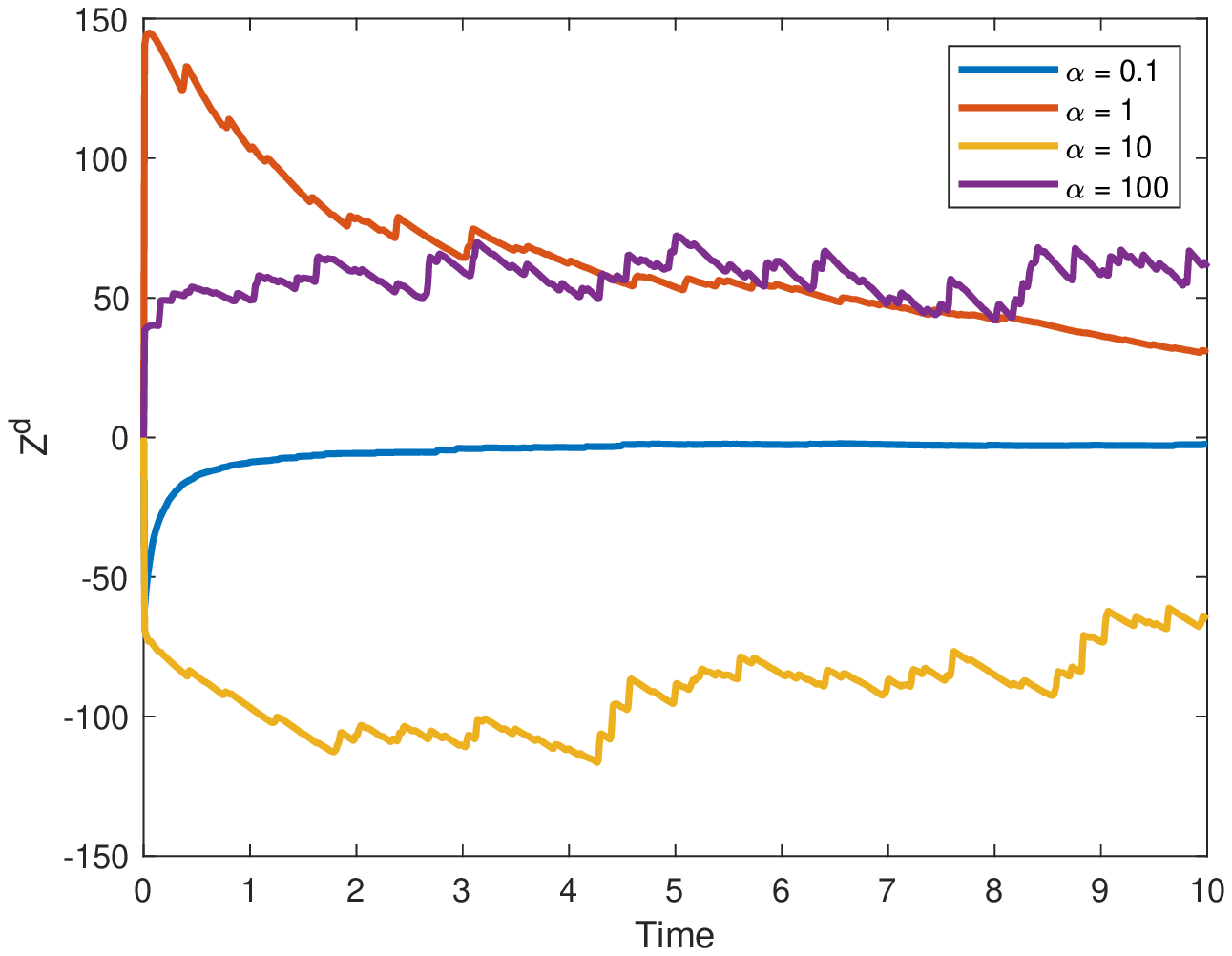}
      \caption{Sample paths $Z^d$ of the limit process for different $\alpha$ and $h= 0.12$ $(a=5, b=5)$.}
    \label{paths3}
\end{figure}

\section*{Acknowledgements}
H\'ector Araya was partially supported by Proyecto Fondecyt PostDoctorado, Chile 3190465, MEC 80190045, Math-Amsud 20-MATH-05 and Redes 190038. Johanna Garz\'on was partially supported by  HERMES project 52433.

\end{document}